\documentclass{amsproc}
\usepackage{amsmath}
\usepackage{amsfonts}
\usepackage{amssymb}
\usepackage{hyperref}
\hypersetup{colorlinks,
citecolor=black,
filecolor=black,
linkcolor=black,
urlcolor=black}
\usepackage{graphics,epstopdf}
\usepackage[pdftex]{graphicx}
\usepackage{newlfont}\newlength{\defbaselineskip}
\usepackage[left=1in,right=1in,top=1in,bottom=0.8in,footskip=0.25in]{geometry}
\usepackage{setspace}
\allowdisplaybreaks
\newtheorem{theorem}{Theorem}[section]
\newtheorem{example}{Example}[section]
\newtheorem{lemma}{Lemma}[section]

\newtheorem{remark}{Remark}[section]

\numberwithin{equation}{section}

\usepackage{cite}

\usepackage{fancyhdr}
\usepackage{graphics}
\usepackage{color,colortbl}
\usepackage{rotating}
\usepackage{fancybox}
\usepackage[table]{xcolor}

\begin{document}
\title{Approximation properties by some modified Sz\'asz-Mirakjan-Kantorovich operators
}
\maketitle
\begin{center}
{\bf Rishikesh Yadav$^{1,\dag}$,  Ramakanta Meher$^{1,\star}$,  Vishnu Narayan Mishra$^{2,\circledast}$}\\
$^{1}$Applied Mathematics and Humanities Department,
Sardar Vallabhbhai National Institute of Technology Surat, Surat-395 007 (Gujarat), India.\\
$^{2}$Department of Mathematics, Indira Gandhi National Tribal University, Lalpur, Amarkantak-484 887, Anuppur, Madhya Pradesh, India\\
\end{center}
\begin{center}
$^\dag$rishikesh2506@gmail.com,  $^\star$meher\_ramakanta@yahoo.com,
 $^\circledast$vishnunarayanmishra@gmail.com
\end{center}

\vskip0.5in

\begin{abstract}
The present article deals with the local approximation results by means of Lipschitz maximal function, Ditzian-Totik modulus of smoothness and Lipschitz type space having two parameters for the summation-integral
type operators defined by Mishra and Yadav (Tbilisi Mathematical Journal. 11(3), (2018), 175-91). Further, we determine the rate of convergence in the term of the the with derivative of bounded variation and for the quantitative means of the defined operators, we establish the quantitative Voronovskaya type and Gr$\ddot{\text{u}}$ss type theorems. Moreover the examples are given with graphical representation to support the main results.
\end{abstract}
\subjclass \textbf{MSC 2010}: {41A25, 41A35, 41A36}.


\textbf{Keywords:}  Rate of convergence, Lipschitz function, Ditzian-Totik modulus of smoothness, function of bounded variation.

\section{Introduction}

To study the approximations properties on unbounded interval, Sz\'asz \cite{40} and Mirakjan \cite{32} introduced the operators known as Sz\'asz-Mirakjan operators. In 1954, Butzer \cite{6} generalized into integral modification of the Sz\'asz-Mirakjan operators known as Sz\'asz-Mirakjan-Kantorovich operators. Totik \cite{41} studied the approximations properties of the Sz\'asz-Mirakjan-Kantorovich operators. Some modifications regarding Kantorovich variant can be seen in various papers such as \cite{11,16,44}. 
Using Brenke-type polynomials, Ta\c{s}delen et al. \cite{42} presented Kantorovich variant operators introduced by Verma et al. \cite{43}. The approximation problems are discussed in many research articles for Kantorovich type operators, such as \cite{5,13,33}.\\

Here, the rate of convergence will be discussed by means of the function with derivative of bounded variation.
First of all, in 1979, Bojanic \cite{7} estimated the rate of convergence for Fourier series while in 1983, this property has been discussed for linear positive operators by Cheng \cite{10}.
Guo and Khan \cite{17} obtained the rate of convergence for some operators using function of bounded variation. Also, Guo \cite{19}  established the rate of convergence for the Durrmeyer operator independently. Later on the significant contributions have been seen in 
\cite{8,9}.
 After two years, an important discussion was occurred regarding convergence rates of approximation for functions of bounded variation and for functions with derivatives of bounded variation in \cite{39} by Shaw et al. 
In this direction, many researchers, authors played the significant role to stablish the approximations resuts regarding rate of convergence by means of function of bounded variation and now a days this type of research is being done with much better quality. We refer some important contributions for the reader \cite{4,25,26,29,37,38}.

Also, one of the discussing area of research is quantitative means of Voronovskaya type theorem. In 2006, Gonska et al. \cite{24}, established the quantitatively Voronovskaya type theorem for any linear positive operators on any compact interval using Taylor's formula for the $n^{\text th}$ continuously diferentiable function and obtained an estimate in terms of the least concave majorant of the modulus of continuity. 
In 1935, an equality was developed by Gr$\ddot{\text{u}}$ss, known as Gr$\ddot{\text{u}}$ss inequality on his name, this inequalty shows a relation between  the integral of product and product of integrals of  two functions. An application of this inequality has been seen in approximation theory in 2011 by  Acu et al. \cite{1} and using the Gr$\ddot{\text{u}}$ss inequality for the Bernstein polynomials, Gal and Gonska \cite{25} proved the Gr$\ddot{\text{u}}$ss Voronovskaya type theorem. Gonska and Tachev \cite{22} obtained a new approach of Bernstein's operators on applying Gr$\ddot{\text{u}}$ss type inequalities using the least concave majorant of the first order modulus of continuity. Now, this has been become broad area of research. Recently, Acar \cite{2} obtained the quantitative Voronokskaya and Gr$\ddot{\text{u}}$ss Voronokskaya type results for the Sz\'asz operators in quantum calculus. We refer some papers which have significant contributions in this regard as \cite{3,14,15,18,35}.

 
Motivated by the above works, we study the approximation properties for the operators defined by Mishra and Yadav \cite{34}. They introduced some modified Sz\'asz-Mirakjan- Kantorovich operators.
Direct results and weighted approximation properties have been discussed as well as they determined the rate of convergence and the comparison took place with the Sz\'asz-Mirakjan-Kantorovich operators by graphical analysis. The modified operators are as given below:

%
%

\begin{eqnarray}\label{RV}
\Hat{\mathcal{R}}_{m,a}(f;x)&=& m\sum\limits_{k=0}^{\infty}s_{m}^a(x) ~\int\limits_{\frac{k}{m}}^{\frac{k+1}{m}}f(t)~dt,
\end{eqnarray}

where $s_{m}^a(x)=a^{\left(\frac{-x}{-1+a^{\frac{1}{m}}}\right)}\frac{x^{k}(\log{a})^{k}}{(-1+a^{\frac{1}{m}})^{k}k!}$,
 $m\in\mathbb{N}$, $x\in X$
  and $a>1$(fixed).\\
 In this regard, we shall further investigate other properties of above operators (\ref{RV}) for approximations point of view. The main aim of this article is to investigate the approximation properties like as rate of convergence in the term of function with derivative of bounded variations, local approximations properties including order of approximation in terms of Lipschitz Maximal function, Ditzian-Totik modulus of smoothness, Peetre's $K$-functional and in a new type of Lipschitz-space having two parameters. Next section consists a  quantitative approximation and additionally quantitative Voronovskaya type, Gr$\ddot{\text{u}}$ss Voronovskaya type theorems are established. 
Also, we study the graphical analysis of proposed operators in last section.

Here, we point out some basic lemmas, which are used to prove our main theorem. Let us define the function $e_i=x^i$, where $i=0,1,2,3$, then we have the following lemma. 
\begin{lemma}\cite{34}\label{L1}
For every $x\in[0,\infty)$ and $a>1$ fixed, it holds

\begin{eqnarray*}
&& 1.~\Hat{\mathcal{R}}_{m,a}(e_0;x)=1,\\
&& 2.~\Hat{\mathcal{R}}_{m,a}(e_1;x)=\frac{1}{2m}+\frac{x\log{a}}{\left(-1+a^{\frac{1}{m}}\right)m},\\ 
&& 3.~\Hat{\mathcal{R}}_{m,a}(e_2;x)=\frac{1}{3m^2}+\frac{2x\log{a}}{\left(-1+a^{\frac{1}{m}}\right)m^2}+\frac{x^2(\log{a})^2}{\left(-1+a^{\frac{1}{m}}\right)^2m^2},\\
&& 4.~\Hat{\mathcal{R}}_{m,a}(e_3;x)=\frac{1}{4m^3}+\frac{7}{2}\frac{x\log{a}}{\left(-1+a^{\frac{1}{m}}\right)m^3}+\frac{9}{2}\frac{x^2(\log{a})^2}{\left(-1+a^{\frac{1}{m}}\right)^2m^3}+\frac{x^3(\log{a})^3}{\left(-1+a^{\frac{1}{m}}\right)^3m^3}.
\end{eqnarray*}

\end{lemma}
Consider, $\Lambda_t^m(x)=\Hat{\mathcal{R}}_{m,a}(\xi_x^m(t);x)$ are known as central moments, where $\xi_x^m(t)=(t-x)^m$, $m=1,2,3,4$ then by Lemma \ref{L1}, following results are obtained. 
 \begin{lemma}\cite{34}\label{L2}
For every $x\geq0$, we have

\begin{eqnarray*}
&&1.~\Lambda_t(x)=-\frac{(-1+2mx)}{2m}+\frac{x\log{a}}{m(-1+a^{\frac{1}{m}})},\\
&&2.~ \Lambda_t^2(x)=\frac{(1-3mx+3m^2x^2)}{3m^2}-\frac{2(-1+a^{\frac{1}{m}})(-1+mx)x\log{a}}{\left(-1+a^{\frac{1}{m}}\right)^2m^2}+\frac{x^2(\log{a})^2}{\left(-1+a^{\frac{1}{m}}\right)^2m^2},\\
&&3.~\Lambda_t^3(x) = -\frac{(-1+4mx-6m^2x^2+4m^3x^3)}{4m^3}+\frac{x(7-12mx+6m^2x^2)\log{a}}{2\left(-1+a^{\frac{1}{m}}\right) m^3} \\&&\hspace{2 cm}-\frac{3x^2(-3+2mx)(\log{a})^2+4x^3(\log{a})^3}{2\left(-1+a^{\frac{1}{m}}\right)^2m^3},\\
&&4.~\Lambda_t^4(x) = \frac{1}{5\left(-1+a^{\frac{1}{m}} \right)^4m^4}(\left(-1+a^{\frac{1}{m}}\right)^4(1-5mx+10m^2x^2-10m^3x^3+5m^4x^4) \\&&\hspace{3 cm}-10\left(-1+a^{\frac{1}{m}}\right)^3 x(-3+7mx-6m^2x^2+2m^3x^3)\log{a}\\&&\hspace{3 cm}+15\left(-1+a^{\frac{1}{m}}\right)^2 x^2 (5-6mx+2m^2x^2)(\log{a})^2\\&&\hspace{3 cm}-20\left(-1+a^{\frac{1}{m}}\right) x^3(-2+mx)(\log{a})^3 +5x^4(\log{a})^4).
\end{eqnarray*}

\end{lemma}

%
\begin{lemma}\label{l2}
For all $x\geq 0$, then there exist a positive $C$ for which, we have following inequalities:

\begin{eqnarray*}
\Lambda_t^1(x)&\leq & \frac{1}{2m},\\
\Lambda_t^2(x)&\leq & \frac{C}{m}x(x+1).
\end{eqnarray*}


\end{lemma}
\begin{proof}
For $m\in\mathbb{N}$, we have

\begin{eqnarray}
\Lambda_t^1(x)& = & \frac{1}{2m}-x+\frac{x\log{a}}{\left(-1+a^{\frac{1}{m}} \right)m}\nonumber\\
&\leq &  \frac{1}{2m}-x+x=\frac{1}{2m}\nonumber\\
\Lambda_t^2(x)&=&\frac{1-3mx + 3m^2x^2}{3m^2}-\frac{2(mx-1)x\log{a}}{\left(-1+a^{\frac{1}{m}} \right)m^2}+x^2\left(\frac{\log{a}}{\left(-1+a^{\frac{1}{m}} \right)m}\right)^2\nonumber\\
&= & \frac{1}{3m^2}-\frac{x}{m}+\left(\frac{\log{a}}{\left(-1+a^{\frac{1}{m}} \right)m}-1\right)^2x^2+\frac{2x}{m} \left(\frac{\log{a}}{\left(-1+a^{\frac{1}{m}} \right)m}\right)\nonumber\\
&\leq & \frac{1}{3m^2}-\frac{x}{m}+\frac{x^2}{m}+\frac{2x}{m}\nonumber\\
&\leq & \frac{1}{3m^2}+\frac{x(x+1)}{m}\leq \frac{C}{m}x(x+1).\nonumber
\end{eqnarray}

\end{proof}

\begin{lemma}\label{l1}
For each $x\geq0$, one can obtain

\begin{enumerate}
\item $\underset{m\to\infty}{\lim} m\Lambda_t^2(x)=x$
\item $\underset{m\to\infty}{\lim} m^2\Lambda_t^3(x)=-\frac{1}{2}x(3x\log{a}-5)$
\item $\underset{m\to\infty}{\lim} m^3\Lambda_t^6(x)=15x^3$.
\end{enumerate}

\end{lemma}
\begin{proof}
Using the Lemma \ref{L2}, we can write as:

\begin{eqnarray*}
\underset{m\to\infty}\lim m\Lambda_t^2(x)&=&\underset{m\to\infty}\lim\frac{(1-3mx+3m^2x^2)}{3m}-\frac{2(-1+a^{\frac{1}{m}})(-1+mx)x\log{a}}{\left(-1+a^{\frac{1}{m}}\right)^2m}+\frac{x^2(\log{a})^2}{\left(-1+a^{\frac{1}{m}}\right)^2m}\\
&=& \underset{m\to\infty}\lim \frac{(1-3mx+3m^2x^2)\left(-1+a^{\frac{1}{m}}\right)^2-6(-1+a^{\frac{1}{m}})(-1+mx)x\log{a}+3x^2(\log{a})^2}{3\left(-1+a^{\frac{1}{m}}\right)^2m}\\
&=& \underset{m\to\infty}\lim \frac{\left(\frac{1}{m^2}-\frac{3x}{m}+3x^2\right)\left(-1+a^{\frac{1}{m}}\right)^2-6(-1+a^{\frac{1}{m}})\left(-\frac{1}{m^2}+\frac{x}{m}\right)x\log{a}+\frac{3x^2}{m^2}(\log{a})^2}{3\left(-1+a^{\frac{1}{m}}\right)^2\frac{1}{m}}=I (\text{say})
\end{eqnarray*}
Replacing $\frac{1}{m}$ by $l$, we have
\begin{eqnarray}
I&=&\underset{l\to 0}\lim\frac{\left(a^l-1\right)^2 \left(l^2-3 l x+3 x^2\right)-6 x \left(a^l-1\right) \log{a} \left(l x-l^2\right)+3 l^2 x^2 (\log{a})^2}{3l (-1 + a^l)^2},~~~\left(\frac{0}{0}~\text{form} \right).
\end{eqnarray}
Using three times L'Hospital rule for the above limit, we obtain
\begin{eqnarray*}
I&=&\underset{l\to 0}\lim \frac{P}{Q}
\end{eqnarray*}
where \begin{eqnarray*}
P &=& 2 a^l \log (a) \bigg((\log{a})^2 \left(l^2 \left(4 a^l-1\right)+12 x^2 \left(a^l-1\right)-3 l x \left(4 a^l-7\right)\right)+3 \log (a) \left(-6 x a^l+l \left(4 a^l-2\right)+9 x\right)\\
&&+6 \left(a^l-1\right)+3 l x (\log{a})^2 (l-x)\bigg)\\
Q&=& 6 a^l (\log{a})^2 \left(6 a^l+l \left(4 a^l-1\right) \log (a)-3\right),
\end{eqnarray*}
and then \begin{eqnarray}
\underset{l\to 0}\lim P &=& 18 x (\log{a})^2\\
\underset{l\to 0}\lim Q &=& 18 (\log{a})^2,
\end{eqnarray}
and hence $I=x$. Similarly, we can prove other parts.
\end{proof}

\begin{theorem}\label{th1}
If $g\in C_B[0,\infty)$ then for all $m\in \mathbb{N}$, it holds
\begin{eqnarray*}
\underset{m\to\infty}\lim \Hat{\mathcal{R}}_{m,a}(g;x)=g(x),
\end{eqnarray*}
uniformly on every compact subset of $[0,\infty)$.
\end{theorem}

\begin{remark}\label{rem1}
If $g$ be a continuous and bounded function on $[0,\infty)$ with supremum norm as $\|g\|=\underset{x\geq 0}{\sup}|g(x)|$ then
\begin{eqnarray*}
|\Hat{\mathcal{R}}_{m,a}(g;x)| & \leq & \|g\|.
\end{eqnarray*}

\end{remark}
\begin{remark}
One may write the above operators into integral representation as 
\begin{eqnarray}\label{NRV}
\Hat{\mathcal{R}}_{m,a}(g;x)=\int\limits_0^\infty \mathfrak{R}(x,t)g(t)~dt,
\end{eqnarray}
where $\mathfrak{R}(x,u)=m\sum\limits_{k=0}^\infty s_{m}^a(x)\chi_{m,k}(x,t)$, where  $\chi_{m,k}(x,t)$ is the characteristic function of the interval $\left[\frac{k}{m},\frac{k+1}{m} \right]$ with respect to $[0,\infty)$.
\end{remark}

\section{Local results}

This section deals with the local approximation properties for the defined operators. Here, we determine the rate of convergence by means of some spaces known as Lipschitz Maximal space defined by Lenze \cite{30} in 1988, with order $\alpha\in(0,1]$ and it can be defined as follows:
\begin{eqnarray}
\eta_\alpha(f;x)=\underset{t,x\geq0}\sup\frac{|f(t)-f(x)|}{|t-x|^\alpha},~t\neq x.
\end{eqnarray}
Here, an upper bound can be obtained for the defined operators (\ref{RV}) with the function in the terms of Lipschitz Maximal function.

\begin{theorem}
For $g\in C_B[0,\infty)$ and for every $\geq0$, we obtain
\begin{eqnarray}
|\Hat{\mathcal{R}}_{m,a}(g;x)-g(x)|\leq \eta_\alpha(g;x)\sqrt{\Lambda_t^2(x)},
\end{eqnarray}
where $\Lambda_t^2(x)=\frac{(1-3mx+3m^2x^2)}{3m^2}-\frac{2(-1+a^{\frac{1}{m}})(-1+mx)x\log{a}}{\left(-1+a^{\frac{1}{m}}\right)^2m^2}+\frac{x^2(\log{a})^2}{\left(-1+a^{\frac{1}{m}}\right)^2m^2}$.
\end{theorem}
\begin{proof}
Using the definition of Lipschitz Maximal space and applying the defined operators (\ref{RV}), we get
\begin{eqnarray}
|\Hat{\mathcal{R}}_{m,a}(g;x)-g(x)|\leq \eta_\alpha(g;x)\Hat{\mathcal{R}}_{m,a}(|t-x|^\alpha;x).
\end{eqnarray}
Using Lemma \ref{L2} and applying H\"older inequality with $p=\frac{2}{\alpha}$ and $p=\frac{2}{2-\alpha}$, we get
\begin{eqnarray*}
|\Hat{\mathcal{R}}_{m,a}(g;x)-g(x)|&\leq & \eta_\alpha(g;x)\left(\Hat{\mathcal{R}}_{m,a}(|t-x|^2;x)\right)^{\frac{1}{2}}\\
&=& \eta_\alpha(g;x)\sqrt{\Lambda_t^2(x)}.
\end{eqnarray*}
Hence, the proof is completed.
\end{proof}
Now, we find the order of approximation for the defined operators (\ref{RV}) in terms of Ditzian-Totik modulus of smoothness. So, consider the function $g\in C_B[0,\infty)$ for which, the Ditzian-Totik modulus of smoothness is defined by

\begin{eqnarray}
 \varpi_\psi(g;\epsilon)=\underset{h\in(0,\epsilon]}\sup \left\{\left|g\left(x+\frac{h\psi(x)}{2} \right)-g\left(x-\frac{h\psi(x)}{2} \right)\right|;x\pm \frac{h\psi(x)}{2}\in(0,\infty) \right\},
\end{eqnarray}
where $\psi(x)=\left(x(1+x) \right)^{\frac{1}{2}}$ and the appropriate $K$-functional can be defined by
\begin{eqnarray}
K_\psi(g;\epsilon)=\underset{f\in \mathcal{W}_\psi[0,\infty)}\inf\{ \|f-g\|+\epsilon\|\psi f'\|, \epsilon>0\},
\end{eqnarray}
where $\mathcal{W}_\psi[0,\infty)=\{f:f\in {AC}_{loc}[0,\infty);\|\psi f'\|<\infty\}$, here ${AC}_{loc}[0,\infty)$ is the space of absolutely continuous and differentiable function on every compact interval of $[0,\infty)$. A relation is obtained between Ditzian-Totik
modulus of smoothness and the appropriated $K$-functional from \cite{12}, according to that, there exist a positive constant $M$, such that

\begin{eqnarray}
M^{-1}\varpi_\psi(g;\epsilon)\leq K_\psi(g;\epsilon)\leq M \varpi_\psi(g;\epsilon).
\end{eqnarray}
\begin{theorem}
Consider $g\in C_B[0,\infty)$, $x\geq 0$, it holds
\begin{eqnarray}
|\Hat{\mathcal{R}}_{m,a}(g;x)-g(x)|\leq 2K_\psi\left(g;\frac{u(x)\sqrt{\Lambda_t^2(x)}}{\psi(x)}\right),
\end{eqnarray}
where $u(x)=\sqrt{x}+\sqrt{1+x}$ and $\Lambda_t^2(x)$ can be obtained by the Lemma \ref{L2}.
\end{theorem}
\begin{proof}
Using Taylor’s theorem by considering  $f\in \mathcal{W}_\psi[0,\infty)$, we have 
\begin{eqnarray*}
f(t)-f(x)&=&\int\limits_x^tf'(l)dl=\int\limits_x^t\frac{f'(l)\psi(l)}{\psi(l)} dl\\
f(t)-f(x)&\leq & \|f'\psi\|\left|\int\limits_x^t\frac{1}{\psi(l)} dl\right|\\
&=& \|f' \psi\| \left|\int\limits_x^t \frac{1}{\sqrt{l(l+1)}} dl\right|\leq \|f' \psi\| \left|\int\limits_x^t \left(\frac{1}{\sqrt{l}}+\frac{1}{\sqrt{l+1}}\right) dl \right|\\
&= & 2\|f' \psi\| \left|\left[ \sqrt{l}+\sqrt{l+1}  \right]_x^t \right|\\
&=& 2\|f' \psi\| \left| \sqrt{t}-\sqrt{x}+\sqrt{t+1}-\sqrt{x+1}\right|\\
&=& 2|t-x|\|f' \psi\|\left(\frac{1}{\sqrt{t}+\sqrt{x}}+\frac{1}{\sqrt{t+1}+\sqrt{x+1}} \right)\\
&\leq & 2|t-x|\|f' \psi\| \left(\frac{\sqrt{x}+\sqrt{x+1}}{\sqrt{x\sqrt{x+1}}} \right)\leq  2|t-x|\|f' \psi\|\frac{u(x)}{\psi(x)}.
\end{eqnarray*}
Using Lemma \ref{L1}, Remark \ref{rem1} and by the above inequality, we  can write as
\begin{eqnarray*}
|\Hat{\mathcal{R}}_{m,a}(g;x)-g(x)|&\leq &|\Hat{\mathcal{R}}_{m,a}(g-f;x)|+\left| \Hat{\mathcal{R}}_{m,a}(f;x)-f(x)\right|+|f(x)-g(x)|\\
&\leq & \Hat{\mathcal{R}}_{m,a}(|g-f|;x)+\Hat{\mathcal{R}}_{m,a}(|f(t)-f(x)|;x)+\|g-f\|\\
&\leq & 2\|g-f\|+2\|f' \psi\|\frac{u(x)}{\psi(x)}\Hat{\mathcal{R}}_{m,a}(|t-x|;x)\\
&\leq & 2\|g-f\|+2\|f' \psi\|\frac{u(x)}{\psi(x)}\left(\Hat{\mathcal{R}}_{m,a}((t-x)^2;x)\right)^{\frac{1}{2}}\\
&=& 2\|g-f\|+2\|f' \psi\|\frac{u(x)}{\psi(x)}\left(\Lambda_t^2(x)\right)^{\frac{1}{2}},
\end{eqnarray*}
taking infimum on right side over all $f\in \mathcal{W}_\psi[0,\infty)$, we get
\begin{eqnarray}
|\Hat{\mathcal{R}}_{m,a}(g;x)-g(x)|\leq 2K_\psi\left(g;\frac{u(x)\sqrt{\Lambda_t^2(x)}}{\psi(x)}\right).
\end{eqnarray}
Thus, the proof is completed.
\end{proof}
\"Ozarslan and Aktu\u{g}lu \cite{36} defined a new type of Lipschitz-space having two parameters. Let $u,v>0$ be fixed numbers, then Lipschitz-type-space is defined by:
\begin{eqnarray}
Lip_M^{u,v}(a)=\left\{g\in C[0,\infty):|g(y)-g(x)|\leq M \frac{|y-x|^a}{(y+ux^2+vx)^{\frac{a}{2}}};x,y\in[0,\infty) \right\},~~a\in(0,1].
\end{eqnarray}
Using the above definition, we have the local approximation result:
\begin{theorem}
Let $g\in Lip_M^{u,v}(a)$ with $a\in(0,1]$ then for every $x\geq 0$, it holds:
\begin{eqnarray}
|\Hat{\mathcal{R}}_{m,a}(g;x)-g(x)|\leq M \left( \frac{\Lambda_t^2(x)}{{ux^2+vx}}\right)^{\frac{\alpha}{2}}.
\end{eqnarray}
\end{theorem}
\begin{proof}
We prove the above theorem within case for $a\in(0,1]$. So, far that, consider\\
\textbf{Case 1.} when $a=1$, then 
\begin{eqnarray*}
|\Hat{\mathcal{R}}_{m,a}(g;x)-g(x)|&\leq &\Hat{\mathcal{R}}_{m,a}(|g(t)-g(x);x)\\
&\leq & M \Hat{\mathcal{R}}_{m,a}\left(\frac{|t-x|}{(y+ux^2+vx)^{\frac{1}{2}}} \right)\\
&\leq & \frac{M}{(ux^2+vx)^{\frac{1}{2}}}\Hat{\mathcal{R}}_{m,a}(|t-x|;x)\\
&\leq & \frac{M}{(ux^2+vx)^{\frac{1}{2}}}\left(\Hat{\mathcal{R}}_{m,a}((t-x)^{2};x\right)^{\frac{1}{2}}\\
&=&\frac{M\sqrt{\Lambda_t^2(x)}}{(ux^2+vx)^{\frac{1}{2}}}.
\end{eqnarray*}
\textbf{Case 2.} when $a\in(0,1)$ the
\begin{eqnarray*}
|\Hat{\mathcal{R}}_{m,a}(g;x)-g(x)|&\leq &\Hat{\mathcal{R}}_{m,a}(|g(t)-g(x);x)\\
&\leq & M \Hat{\mathcal{R}}_{m,a}\left(\frac{|t-x|^{a}}{(y+ux^2+vx)^{\frac{a}{2}}} \right)\\
&\leq & \frac{M}{(ux^2+vx)^{\frac{a}{2}}}\Hat{\mathcal{R}}_{m,a}(|t-x|^a;x).
\end{eqnarray*}
Let $p=\frac{1}{a},~q=\frac{1}{1-a}$ and applying H\"older inequality, we get
\begin{eqnarray*}
|\Hat{\mathcal{R}}_{m,a}(g;x)-g(x)|&\leq & \frac{M} {(ux^2+vx)^{\frac{a}{2}}}\left( \Hat{\mathcal{R}}_{m,a}(|t-x|;x)\right)^a\\
&\leq & \frac{M}{(ux^2+vx)^{\frac{a}{2}}} \left( \Hat{\mathcal{R}}_{m,a}((t-x)^2;x)\right)^{\frac{a}{2}}\\
&=& M\left( \frac{\Lambda_t^2(x)}{ux^2+vx}\right)^{\frac{a}{2}}
\end{eqnarray*}
hence, the required result is obtained.
\end{proof}

\section{Rate of convergence in term of the derivative of bounded variation}

Now we determine the rate of convergence of the said operators  in the space of the function of bounded variation by considering $DBV[0,\infty)$, the set of all continuous function having derivative of bounded variation on every finite sub-interval of the $[0,\infty)$. One can observe that for each $g\in DBV[0,\infty)$, it can be written
\begin{scriptsize}
\begin{eqnarray}
g(x)=\int\limits_0^x h(s)~ds+g(0),
\end{eqnarray}
\end{scriptsize}
where $h$ is a function bounded of variation on each finite sub-interval of $[0,\infty)$. To determine the rate of convergence in the terms of function of bounded variation, for which the function $g\in DBV[0,\infty)$, we use an auxiliary operators $g_x$ such that
\begin{eqnarray}\label{eq5}
g_x(t) &= &  
\begin{cases}
    g(t)-g(x-),& 0\leq t<x,\\
    0,& t=x,\\
    g(t)-g(x+), &  x<t<\infty.
\end{cases} 
\end{eqnarray}

Moreover, we denote $V_a^b g$ as total variation of a real valued function $g$ defined on $[a,b]\subset[0,\infty)$ with the quantity
\begin{eqnarray}
V_a^b g=\underset{\mathcal{S}}\sup\left(\sum\limits_{j=0}^{n_{P}-1}|g(x_{j+1})-g(x_j)| \right),
\end{eqnarray} 
where, $\mathcal{S}$ is the set of all partition $P=\{a=x_0,\cdots,x_{n_P}=b\}$ of the interval $[a,b]$. 

\begin{lemma}\label{l3}
For all $x\in[0,\infty)$, $n\in\mathbb{N}$ there exist a positive constant $M>0$, it can be written as

\begin{enumerate}
\item{} $\mathrm{J}(x,y)= \int\limits_0^y \mathfrak{R}(x,t)~dt\leq \frac{C}{m(x-y)^2}(x(x+1)),~~0\leq y<x,$
\item{} $1-\mathrm{J}_n(x,z)= \int\limits_z^\infty \mathfrak{R}(x,t)~dt\leq \frac{C}{m(z-x)^2}(x(x+1)),~~~~x<z<\infty.$
\end{enumerate}

\end{lemma}
\begin{proof}
Here, $0\leq y<x$ and $x\geq0$ then we have

\begin{eqnarray*}
\int\limits_0^y \mathfrak{R}(x,t)~dt &\leq & \int\limits_0^y\left( \frac{x-t}{x-y}\right)^2\mathfrak{R}(x,t) ~dt\\
&\leq & \frac{\Lambda_t^2(x)y}{(x-y)^2}\\
&\leq & \frac{C}{(x-y)^2}\frac{x(x+1)}{m}.
\end{eqnarray*}

Similarly, other inequality can be proved.
\end{proof}
\begin{theorem}
Let $f\in DBV[0,\infty)$ and  $\forall~x\in[0,\infty)$, it holds

\begin{eqnarray*}
|\Hat{\mathcal{R}}_{m,a}(f;x)-f(x)|&\leq &  \frac{1}{4m}|f'(x+)+f'(x-)|+ \left(\frac{C}{4m}x(x+1)\right)^{\frac{1}{2}}|f'(x+)-f'(x-)|+\frac{C(x+1)}{m} \sum\limits_{k=0}^{[\sqrt{m}]}\left(V_{x-\frac{x}{k}}^sf'_x \right)\\
&&+\frac{x}{\sqrt{m}}\left(V_{x-\frac{x}{\sqrt{m}}}^xf'_x \right)
+ \frac{x}{\sqrt{m}} V_x^{x+\frac{x}{\sqrt{m}}}(f'_x)
+ \frac{C(x+1)}{m} \sum\limits_{k=0}^{[\sqrt{m}]} V_x^{x+\frac{x}{k}}(f'_x).
\end{eqnarray*}

\end{theorem}
\begin{proof}
Using the hypothesis (\ref{eq5}), we have

\begin{eqnarray}\label{bv1}
f'(s)&=& f_x'(s)+\frac{1}{2}(f'(x+)+f'(x-))+\frac{1}{2}(f'(x+)-f'(x-))\text{sgn}(s-x)\nonumber\\
&&+\sigma_x(s)\{f'(s)-\frac{1}{2}(f'(x+)+f'(x-)) \},
\end{eqnarray}

where $\sigma_x(s)$
\begin{eqnarray}\label{bv2}
\sigma_x(s)=
\begin{cases}
1 & s=x\\
0 & s\neq x.
\end{cases}
\end{eqnarray}
By equation (\ref{NRV}), it can be written as

\begin{eqnarray}\label{bv3}
\Hat{\mathcal{R}}_{m,a}(f;x)-f(x)&=& \int\limits_0^\infty \mathfrak{R}(x,s)(f(s)-f(x))~ds\nonumber\\
&=&  \int\limits_0^\infty \mathfrak{R}(x,s)\left(\int\limits_x^s f'(t)~dt\right)ds.
\end{eqnarray}

Here, it is clear that
\begin{eqnarray*}
\int\limits_x^s \sigma_x(t)~dt=0,
\end{eqnarray*}
therefore

\begin{eqnarray}\label{bv4}
\int\limits_0^\infty \mathfrak{R}(x,s)\int\limits_x^s\left(\sigma_x(t)\{f'(t)-\frac{1}{2}(f'(x+)+f'(x-))\}dt \right)ds=0.
\end{eqnarray}

By (\ref{NRV}), we yield

\begin{eqnarray}
\int\limits_0^\infty \mathfrak{R}(x,s)\left(\int\limits_x^s\frac{1}{2}(f'(x+)+f'(x-))~dt\right)ds &=&\frac{1}{2}(f'(x+)+f'(x-)) \int\limits_0^\infty \mathfrak{R}(x,s)(s-x)~ds\nonumber\\&=& \frac{1}{2}(f'(x+)+f'(x-)) \Lambda_s(x).
\end{eqnarray}

And 

\begin{eqnarray}\label{bv6}
\left|\int\limits_0^\infty \mathfrak{R}(x,s)\left(\frac{1}{2}\int\limits_x^s(f'(x+)-f'(x-))\text{sgn}(t-x)~dt\right)ds\right| &\leq & \frac{1}{2}|(f'(x+)-f'(x-))|\int\limits_0^\infty \mathcal{Y}_n^{[\alpha]}(x;s)|s-x|~ds\nonumber\\
&\leq & \frac{1}{2}|(f'(x+)-f'(x-))| \Hat{\mathcal{R}}_{m,a}(|s-x|;x)\nonumber\\
&\leq & \frac{1}{2}|(f'(x+)-f'(x-)|\left(\Lambda_s^2(x) \right)^{\frac{1}{2}}
\end{eqnarray}

Using Lemmas (\ref{l1}, \ref{l2}), we have following inequality holds:

\begin{eqnarray}\label{bv7}
|\Hat{\mathcal{R}}_{m,a}(f;x)-f(x)|&\leq &  \frac{1}{2}|f'(x+)+f'(x-)|\Lambda_t^1(x)+ \frac{1}{2}|f'(x+)-f'(x-)|\left(\Lambda_s^2(x)\right)^{\frac{1}{2}}\nonumber\\
&&+ \left|\int\limits_0^\infty \mathfrak{R}(x,s)\left(\frac{1}{2}\int\limits_x^s(f'_x(t))~dt\right)ds\right|
\end{eqnarray}

Here,
 
   \begin{eqnarray}\label{bv8}
   \nonumber \int\limits_0^\infty \mathfrak{R}(x,s)\left(\int\limits_x^s(f'_x(u))~du\right)ds &=& \int\limits_0^x \mathfrak{R}(x,s)\left(\int\limits_x^s(f'_x(t))~dt\right)ds+\int\limits_x^s \mathfrak{R}(x,s)\left(\int\limits_x^s(f'_x(t))~dt\right)ds\\
   &=&L_1+L_2,
   \end{eqnarray}

where 

\begin{eqnarray}
L_1 &=& \int\limits_0^x \left(\int\limits_x^s(f'_x(t))~dt\right)\frac{\partial}{\partial s}(\mathrm{J}(x,s))ds\nonumber\\
&=& \int\limits_0^x f'_x(s)\mathrm{J}(x,s)ds\nonumber\\
&=& \int\limits_0^y f'_x(s)\mathrm{J}(x,s)ds+\int\limits_y^x f'_x(s)\mathrm{J}(x,s)ds
\end{eqnarray}

Here, we consider $y=x-\frac{x}{\sqrt{m}}$ then by above equality, one can write

\begin{eqnarray}
\left|\int\limits_{x-\frac{x}{\sqrt{m}}}^x f'_x(s)\mathrm{J}(x,s)ds \right|&\leq & \int\limits_{x-\frac{x}{\sqrt{m}}}^x |f'_x(s)||\mathrm{J}(x,s)|ds\nonumber\\
&\leq & \int\limits_{x-\frac{x}{\sqrt{m}}}^x |f'_x(s)-f'_x(x)|ds,~~~~f'_x(x)~=0, (\text{where}~|\mathrm{J}(x,s)|\leq 1)\nonumber\\
&\leq & \int\limits_{x-\frac{x}{\sqrt{n}}}^x V_s^xf'_x ds\nonumber\\&\leq & V_{x-\frac{x}{\sqrt{m}}}^xf'_x\int\limits_{x-\frac{x}{\sqrt{m}}}^x ds\nonumber\\
&=&\frac{x}{\sqrt{m}}\left(V_{x-\frac{x}{\sqrt{m}}}^xf'_x \right)
\end{eqnarray}

Using  Lemma \ref{l3} for solving second term,  we get

\begin{eqnarray}
\int\limits_x^{x-\frac{x}{\sqrt{m}}} |f'_x(s)|\mathrm{J}(x,s)ds &\leq & C\frac{x(x+1)}{m} \int\limits_x^{x-\frac{x}{\sqrt{m}}} \frac{|f'_x(s)|}{(x-s)^2} ds\nonumber\\
&\leq & C\frac{x(x+1)}{m} \int\limits_x^{x-\frac{x}{\sqrt{m}}} V_s^xf'_x\frac{1}{(x-s)^2}\nonumber\\
&=& C\frac{x(x+1)}{xm} \int\limits_x^{\sqrt{m}} V_{x-\frac{x}{p}}^sf'_x dp\nonumber\\
&\leq & C\frac{(x+1)}{m} \sum\limits_{k=0}^{[\sqrt{m}]}\left(V_{x-\frac{x}{k}}^sf'_x \right).
\end{eqnarray}

Hence, 

\begin{eqnarray}
|L_1|\leq \frac{C(x+1)}{m} \sum\limits_{k=0}^{[\sqrt{m}]}\left(V_{x-\frac{x}{k}}^sf'_x \right)+\frac{x}{\sqrt{m}}\left(V_{x-\frac{x}{\sqrt{m}}}^xf'_x \right).
\end{eqnarray}

To solve $L_2$, we reform $L_2$ and integrating by parts, we have

\begin{eqnarray*}
|L_2|&=&\Bigg| \int\limits_x^z\left(\int\limits_x^sf'_x(t)dt \right)\frac{\partial}{\partial s}(1-\mathrm{J}(x,s))ds + \int\limits_z^\infty\left(\int\limits_x^sf'_x(t)dt \right)\frac{\partial}{\partial s}(1-\mathrm{J}(x,s))ds\Bigg| \\
&\leq & \left|\int\limits_x^z\left(\int\limits_x^sf'_x(t)dt \right)\frac{\partial}{\partial s}(1-\mathrm{J}(x,s))ds  \right| + \left| \int\limits_z^\infty\left(\int\limits_x^sf'_x(u)du \right)\frac{\partial}{\partial s}(1-\mathrm{J}(x,s))ds     \right|\\
&=& \Bigg|  \left[\int\limits_x^sf'_x(t)dt (1-\mathrm{J}(x,s))\right]_x^z-\int\limits_x^z f'_x(s)(1-\mathrm{J}(x,s))ds \\
&& +  \left[\int\limits_x^sf'_x(t)dt (1-\mathrm{J}(x,s)) \right]_z^\infty - \int\limits_z^\infty f'_x(s)(1-\mathrm{J}(x,s))  ds\Bigg| \\
&=& \Bigg| \int\limits_x^z f'_x(t)dt (1-\mathrm{J}(x,z)) -\int\limits_x^z f'_x(s)(1-\mathrm{J}(x,s))ds \\
&&  -\int\limits_x^z f'_x(t)dt (1-\mathrm{J}(x,z))-\int\limits_z^\infty f'_x(s)(1-\mathrm{J}(x,s))  ds \Bigg|\\
&=& \Bigg| \int\limits_x^z f'_x(s)(1-\mathrm{J}(x,s))ds +\int\limits_z^\infty f'_x(s)(1-\mathrm{J}(x,s))  ds\Bigg|\\
&\leq &  \int\limits_x^z V_x^s (f'_x) ds+ \frac{Cx(x+1)}{m}\int\limits_z^\infty V_x^s(f'_x) \frac{1}{(s-x)^2}ds\\
&\leq & \frac{x}{\sqrt{m}} V_x^{x+\frac{x}{\sqrt{m}}}(f'_x)+\frac{Cx(x+1)}{m}\int\limits_{x+\frac{x}{\sqrt{m}}}^\infty V_x^s(f'_x) \frac{1}{(s-x)^2}ds.
\end{eqnarray*}

On substituting $s=x\left(1+\frac{1}{\eta} \right)$, we obtain

\begin{eqnarray*}
|L_2|&\leq & \frac{x}{\sqrt{m}} V_x^{x+\frac{x}{\sqrt{m}}}(f'_x)+\frac{C(x+1)}{m}\int\limits_{0}^{\sqrt{m}} V_x^{x+\frac{x}{\eta}}(f'_x) d\eta\\
&\leq & \frac{x}{\sqrt{m}} V_x^{x+\frac{x}{\sqrt{m}}}(f'_x)+ \frac{C(x+1)}{m} \sum\limits_{k=0}^{[\sqrt{m}]}\int\limits_{k}^{\sqrt{k+1}}V_x^{x+\frac{x}{k}}(f'_x) d\eta\\
&=& \frac{x}{\sqrt{m}} V_x^{x+\frac{x}{\sqrt{m}}}(f'_x)+ \frac{C(x+1)}{m} \sum\limits_{k=0}^{[\sqrt{m}]} V_x^{x+\frac{x}{k}}(f'_x).
\end{eqnarray*}

Using the value of $L_1, L_2$ in (\ref{bv8}) and with the help of (\ref{bv7}), we obtain required result.

\end{proof}

\section{Quantitative Approximation}\label{sec4}
In 2007, Ispir \cite{26} proposed the weighted modulus of continuity $\Delta(g;\xi)$  for any $\xi>0$, in the weighted space $C_w^k[0,\infty)$ to estimate the degree of approximation, which is as follows:

\begin{eqnarray}
\Delta(g;\xi)=\underset{0\leq h\leq\xi,~0\leq x\leq\infty}\sup \frac{|g(x+h)-g(x)|}{(1+h^2)(1+x^2)},~~~~~~~g\in C_w^k[0,\infty). 
\end{eqnarray} 
where the weighted space is defined as $C_w^k[0,\infty)=\{g\in C_w[0,\infty),\underset{x\to\infty}\lim\frac{|g(x)|}{w(x)}<+\infty\}$, $C_w[0,\infty)=\{g\in B_w[0,\infty), g~\text{is continuous} \} $, $B_w[0,\infty)=\{g:[0,\infty)\to\mathbb{R}| |g(x)|Mw(x)\},$ here $M$ (depending on the function) is a positive constant and $w(x)=1+x^2$ is a weight function.

 \begin{remark}
For $g\in C_w^k[0,\infty)$
 \begin{eqnarray*}
  \underset{\xi\to 0}\lim\Delta(g;\xi)=0.
 \end{eqnarray*}
 \end{remark}
 On can obtains as, $\Delta(f;\lambda\xi)\leq2(1+\xi^2)(1+\lambda)\Delta(f;\xi),~~\lambda>0$. 
 
Using the  weighted modulus of continuity and defined inequality, one can show that
 \begin{eqnarray}
 \nonumber|g(t)-g(x)|&\leq &(1+x^2)(1+(t-x)^2)\Delta(g;|t-x|)\\
 &\leq & 2\left(1+\frac{|t-x|}{\xi}\right)(1+\xi^2)(1+(t-x)^2)(1+x^2)\Delta(f;|t-x|).
 \end{eqnarray}

As the consequence of the weighted modulus of continuity, we determine the degree of approximation of the operators $\mathcal{U}_{n}^{[\alpha]}(g;x)$ in the weighted space $C_w^k[0,\infty)$.

\subsection{Quantitative Voronovskaya type theorem}
\begin{theorem}\label{th}
Let $\mathrm{g}', \mathrm{g}''\in C_w^k[0,\infty)$ and for sufficiently large value of $m\in\mathbb{N}$, then for each $x\geq0$, it holds

\begin{eqnarray*}
m\left |\Hat{\mathcal{R}}_{m,a}(\mathrm{g};x)-\mathrm{g}(x)-\mathrm{g}'(x)\Lambda_t^1(x)-\frac{\mathrm{g}''(x)}{2!}\Lambda_t^2(x)\right|=O(1)\Delta\left(\mathrm{g};\sqrt{\frac{1}{m}}\right).
\end{eqnarray*}

\end{theorem}

\begin{proof}
By Taylor's expansion, it can be written as:

\begin{eqnarray}
\mathrm{g}(t)-\mathrm{g}(x)=\mathrm{g}'(x)\xi_x(t)+\frac{\mathrm{g}''(x)}{2}\xi_x^2(t)+\zeta(t,x),
\end{eqnarray}

where $\zeta(t,x)=\frac{\mathrm{g}''(\theta)-\mathrm{g}''(x)}{2!}(\theta-x)^2$ and $\zeta\in (t,x)$.
Applying operators (\ref{RV}) on both sides to the above expansion, then one can obtains 

\begin{eqnarray}\label{n1}
m\left|\Hat{\mathcal{R}}_{m,a}(\mathrm{g};x)-\mathrm{g}(x)-\mathrm{g}'(x)\Lambda_t(x)-\frac{\mathrm{g}''(x)}{2}\Lambda_t^2(x)\right|\leq m \Hat{\mathcal{R}}_{m,a}(|\zeta(t,x)|;x).
\end{eqnarray}

Now using the property of weighted modulus of continuity, we get

\begin{eqnarray*}
\frac{\mathrm{g}''(\theta)-\mathrm{g}''(x)}{2} 
&\leq & \left(1+\frac{|s-x|}{\xi} \right)(1+\xi^2)(1+(t-x)^2)(1+x^2)\Delta(f'',\xi)
\end{eqnarray*}

and also 

\begin{eqnarray}
\left|\frac{\mathrm{g}''(\theta)-\mathrm{g}''(x)}{2}\right| &\leq &  
\begin{cases}
    2(1+\xi^2)^2(1+x^2)\Delta(\mathrm{g}'',\xi),& |s-x|<\xi,\\
    2(1+\xi^2)^2(1+x^2)\frac{(s-x)^4}{\xi^4}\Delta(\mathrm{g}'',\xi),& |s-x|\geq\xi.
\end{cases} 
\end{eqnarray}

Now for $\xi\in(0,1)$, we get

\begin{eqnarray}
\left|\frac{\mathrm{g}''(\theta)-\mathrm{g}''(x)}{2}\right| &\leq & 8(1+x^2)\left(1+\frac{(s-x4}{\xi^4}\right)\Delta(\mathrm{g}'',\xi). 
\end{eqnarray}

Hence,
$$(|\zeta(s,x)|;x)\leq 8(1+x^2)\left((s-x)^2+\frac{(s-x)^6}{\xi^4}\right)\Delta(\mathrm{g}'',\xi).$$

Thus, applying the Lemma \ref{l1}, one can obtain

\begin{eqnarray*}
\Hat{\mathcal{R}}_{m,a}(| \zeta(t,x)|;x)&\leq & 8(1+x^2)\Delta(\mathrm{g}'',\xi)\left(\Hat{\mathcal{R}}_{m,a}((s-x)^2;x)+\frac{\Hat{\mathcal{R}}_{m,a}((s-x)^6;x)}{\xi^4}\right) \\
&\leq & 8(1+x^2)\Delta(\mathrm{g}'',\xi) \left(O\left(\frac{1}{m} \right)+\frac{1}{\xi^4} O\left(\frac{1}{m^3} \right) \right),~~\text{as}~m\to\infty.
\end{eqnarray*}

Choose, $\xi=\sqrt{\frac{1}{n}}$, then

\begin{eqnarray}
\Hat{\mathcal{R}}_{m,a}(| \zeta(t,x)|;x)\leq 8 O\left(\sqrt{\frac{1}{m}} \right)\Delta\left(\mathrm{g}'',\sqrt{\frac{1}{m}}\right)(1+x^2).
\end{eqnarray}

Hence, we reach on 

\begin{eqnarray}\label{n2}
m\Hat{\mathcal{R}}_{m,a}(|\zeta(s,x)|;x)=O(1)\Delta\left(\mathrm{g}'', \sqrt{\frac{1}{m}}\right).
\end{eqnarray}

By (\ref{n1}) and (\ref{n2}), we obtain the required result.
\end{proof}
Using the above theorem, we obtain a new asymptotic type formula for the defined operators using two functions from the weighted space $C_w^k[0,\infty)$.
\subsection{Gr$\ddot{\text{u}}$ss Voronovskaya type theorem}
\begin{theorem}\label{th2}
Let $\mu,\nu\in C_w^k[0,\infty)$ then for $\mu', \mu'', \nu', \nu''\in C_w^k[0,\infty)$, it holds

\begin{eqnarray*}
\underset{m\to\infty}\lim m\left(\Hat{\mathcal{R}}_{m,a}(\mu\nu;x)-\Hat{\mathcal{R}}_{m,a}(\mu;x)\Hat{\mathcal{R}}_{m,a}(\nu;x) \right)= x \mu'(x)\nu'(x). 
\end{eqnarray*}

\end{theorem} 
\begin{proof}
By making suitable arrangement and using well known properties of derivative of multiplication of two functions, we get

\begin{eqnarray*}
 m\left(\Hat{\mathcal{R}}_{m,a}(\mu\nu;x)-\Hat{\mathcal{R}}_{m,a}(\mu;x)\Hat{\mathcal{R}}_{m,a}(\nu;x) \right)&=& m\Bigg\{\Bigg(\Hat{\mathcal{R}}_{m,a}(\mu\nu;x)-\mu(x)\nu(x)-(\mu\nu)'\Lambda_t^1(x)-\frac{(\mu\nu)''}{2!}\Lambda_t^2(x)\Bigg)\\
 &&-\mathsf{g}(x)\Bigg(\Hat{\mathcal{R}}_{m,a}(\mu;x)-\mu(x)-\mu'(x)\Lambda_t^1(x)-\frac{\mu''(x)}{2!}\Lambda_t^2(x) \Bigg)\\
 &&-\Hat{\mathcal{R}}_{m,a}(\mu;x)\Bigg(\Hat{\mathcal{R}}_{m,a}(\nu;x)-\nu(x)-\nu'(x)\Lambda_t^1(x)-\frac{\nu''(x)}{2!}\Lambda_t^2(x) \Bigg)\\
 &&+\frac{\nu''(x)}{2!}\Hat{\mathcal{R}}_{m,a}((t-x)^2;x)\left(\mu- \Hat{\mathcal{R}}_{m,a}(\mu;x)\right)+\mu'(x)\nu'(x)\Lambda_t^2(x)\\
 &&+ \nu'(x)\Lambda_t^1(x)\left(\mu- \Hat{\mathcal{R}}_{m,a}(\mu;x)\right) \Bigg\}.
\end{eqnarray*}

For sufficiently large value of $m$, i.e. for $m\to\infty$, with the help of Theorems \ref{th1} and \ref{th}, Lemma \ref{l1}, the proof completed  after talking the limit on both sides to the  above equality.
\end{proof}

\section{Graphical Representation}
In this segment, the graphical approach are shown for the said operators (\ref{RV}) regarding convergence to the given function.

\begin{example}
Consider the function $f(x)=x^2e^{x}$ with $x\in[0,5]$. Here we take the value of $m=10,25,100$ for the given operators ${\mathcal{R}}_{m,a}(f;x)$ and fixed value of $a=2$. Then we obtain, the better  approximation by the said operators as we increase the vale of $m$.
\begin{figure}[h!]
    \centering 
    \includegraphics[width=.42\textwidth]{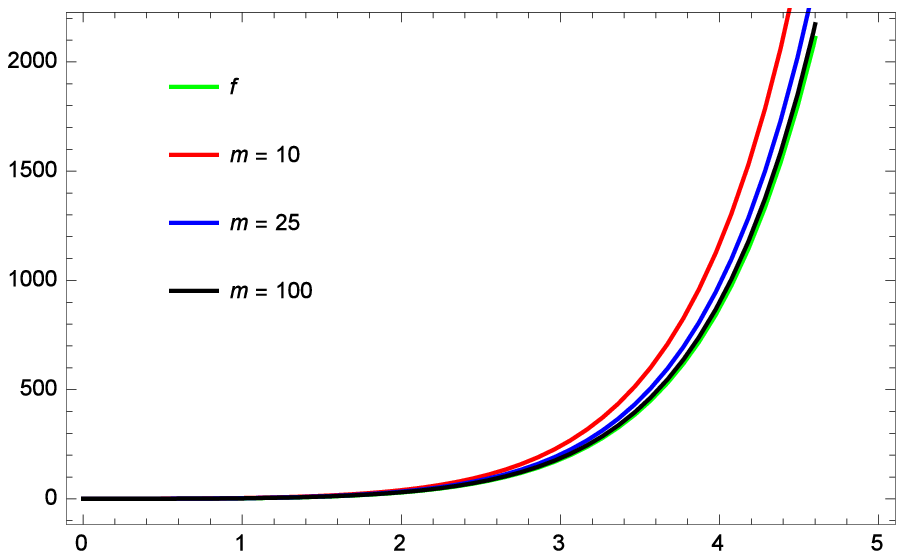}   
    \caption[Description in LOF, taken from~\cite{source}]{The convergence of the operators $\Hat{\mathcal{R}}_{m,a}(f;x)$ to the function $f(x)(green)$.}
    \label{y1}
\end{figure}
\end{example}

\begin{example}
Here, we take the function $f(x)=x\cos{(2x+1)}$ and $x\in[0,5]$, then the approximation of the given function by the said operators take place by graphical representation. Here errors decrease, as the value of $m$ is increased.
\begin{figure}[h!]
    \centering 
    \includegraphics[width=.42\textwidth]{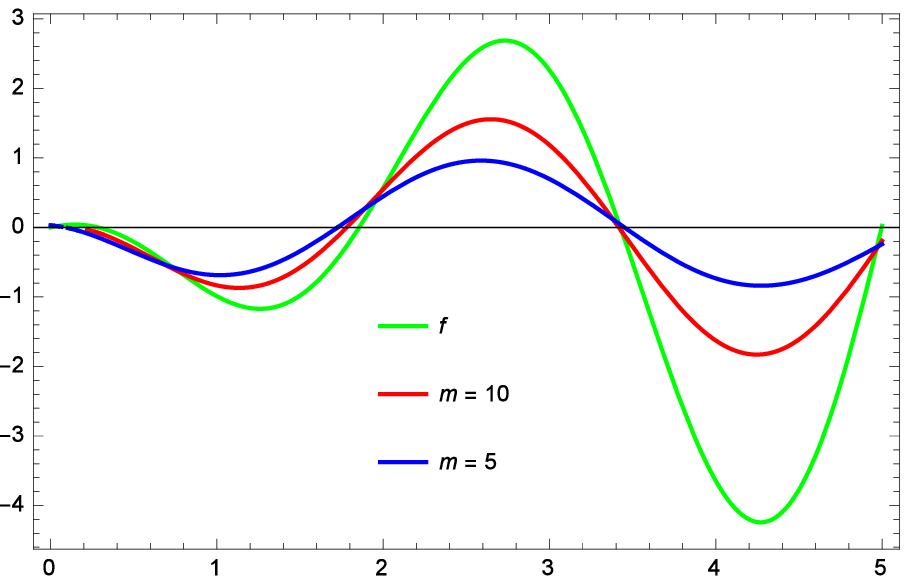}   
    \caption[Description in LOF, taken from~\cite{source}]{The convergence of the operators $\Hat{\mathcal{R}}_{m,a}(f;x)$ to the function $f(x)(green)$.}
    \label{y2}
\end{figure}
\end{example}

\subsection{Concluding Remark}
As we increase the value of $m$, the approximation is good, i.e. for the large value of $m$ the error is minimum.


\begin{thebibliography}{99}





\bibitem{1} Acu AM, Gonska H, Raşa I. Gr\"uss-type and Ostrowski-type inequalities in approximation theory. Ukrainian Mathematical Journal. 2011 Nov 1;63(6):843-64.
\bibitem{2} Acar T. Quantitative $q$-Voronovskaya and $q$-Grüss–Voronovskaya-type results for $q$-Szasz operators. Georgian Mathematical Journal. 2016 Dec 1;23(4):459-68.
\bibitem{3} Acar T, Aral A, Rasa I. The new forms of Voronovskaya’s theorem in weighted spaces. Positivity. 2016 Mar 1;20(1):25-40.
\bibitem{4} Acar T, Gupta V, Aral A. Rate of convergence for generalized Sz\'asz operators. Bulletin of Mathematical Sciences. 2011 Jun 1;1(1):99-113.
\bibitem{5} Altomare F, Montano MC, Leonessa V. On a generalization of Sz\'asz–Mirakjan–Kantorovich operators. Results in Mathematics. 2013 Jun 1;63(3-4):837-63.
\bibitem{6} Butzer PL. On the extensions of Bernstein polynomials to the infinite interval. Proceedings of the American Mathematical Society. 1954 Aug 1;5(4):547-53.
\bibitem{7} Bojanic R. An estimate of the rate of convergence for Fourier series of functions of bounded variation. Publications de l'Institut Math\'ematique [Elektronische Ressource]. 1979;40:57-60.
\bibitem{8} Bojanic R, Cheng F. Rate of convergence of Bernstein polynomials for functions with derivatives of bounded variation. Journal of Mathematical Analysis and Applications. 1989 Jul 1;141(1):136-51.

\bibitem{9} Bojanic R, Khan MK. Rate of convergence of some operators of functions with derivatives of bounded variation. Atti Sem. Mat. Fis. Univ. Modena. 1991;39(2):495-512.
\bibitem{10} Cheng F. On the rate of convergence of Bernstein polynomials of functions of bounded variation. Journal of approximation theory. 1983 Nov 1;39(3):259-74.
\bibitem{11} Duman O, \"Ozarslan MA, Della Vecchia B. Modified Sz\'asz-Mirakjan-Kantorovich operators preserving linear functions. Turkish Journal of Mathematics. 2009 Jun 9;33(2):151-8.
\bibitem{12} Ditzian Z, Totik V. Moduli of Smoothness, Springer Series in Computational Mathematics. Berlin, itd: Springer Verlag. 1987;9.
\bibitem{13} Eren\c{c}in A, Büyükdurakoglu S. A modiffcation of generalized Baskakov-Kantorovich operators. Stud. Univ. Babes-Bolyai, Math. 2014 Sep 1;59:351-64.
\bibitem{14} Erençin A, Raşa I. Voronovskaya type theorems in weighted spaces. Numerical Functional Analysis and Optimization. 2016 Dec 1;37(12):1517-28.
\bibitem{15} Agrawal PN, Baxhaku B, Chauhan R. Quantitative Voronovskaya-and Gr\"uss-Voronovskaya-type theorems by the blending variant of Sz\'asz operators including Brenke-type polynomials. Turkish Journal of Mathematics. 2018 Jul 24;42(4):1610-29.
\bibitem{16} Gupta V, Vasishtha V, Gupta MK. Rate of convergence of the Szasz-Kantorovitch-Bezier operators for bounded variation functions. Publ. Inst. Math.(Beograd)(NS). 2002;72(86):137-43.
\bibitem{17} Guo SS, Khan MK. On the rate of convergence of some operators on functions of bounded variation. Journal of Approximation Theory. 1989 Jul 1;58(1):90-101.
\bibitem{18} Garg T, Acu AM, Agrawal PN. Further results concerning some general Durrmeyer type operators. Revista de la Real Academia de Ciencias Exactas, F\'isicas y Naturales. Serie A. Matem\'aticas. 2019 Jul 1;113(3):2373-90.
\bibitem{19} Guo S. On the rate of convergence of the Durrmeyer operator for functions of bounded variation. Journal of Approximation Theory. 1987 Oct 1;51(2):183-92.
\bibitem{20} Gupta MK, Beniwal MS, Goel P. Rate of convergence for Sz\'asz–Mirakyan–Durrmeyer operators with derivatives of bounded variation. Applied Mathematics and Computation. 2008 Jun 1;199(2):828-32.
\bibitem{21} Gonska H, Pi\c{t}ul P, Ra\c{s}a I. On Peano's form of the Taylor remainder, Voronovskaja's theorem and the commutator of positive linear operators. Univ.; 2006.

\bibitem{22} Gonska H, Tachev G. Gr\"uss-type inequalities for positive linear operators with second order moduli. Matemati\v{c}ki vesnik. 2011 Jan;63(4):247-52.
\bibitem{23} Gonska H, Rasa I. Remarks on Voronovskaya’s theorem. Gen. Math. 2008;16(4):87-99.
\bibitem{24} Gonska H, P\u{a}lt\u{a}nea R. General Voronovskaya and asymptotic theorems in simultaneous approximation. Mediterranean journal of mathematics. 2010 Mar 1;7(1):37-49.
\bibitem{25} Gal S, Gonska H. Gr\"uss and Gr\"uss-Voronovskaya-type estimates for some Bernstein-type polynomials of real and complex variables. arXiv preprint arXiv:1401.6824. 2014 Jan 27.
\bibitem{26} Ispir N. Rate of convergence of generalized rational type Baskakov operators. Mathematical and computer modelling. 2007 Sep 1;46(5-6):625-31.
\bibitem{28} Kajla A, Acu AM, Agrawal PN. Baskakov–Sz\'asz-type operators based on inverse P\'olya–Eggenberger distribution. Annals of Functional Analysis. 2017;8(1):106-23.
\bibitem{29} Karsli H. Rate of convergence of new Gamma type operators for functions with derivatives of bounded variation. Mathematical
and computer modelling. 2007 Mar 1;45(5-6):617-24.
\bibitem{30} Lenze B. On Lipschitz-type maximal functions and their smoothness spaces. In Indagationes Mathematicae (Proceedings) 1988 Jan 1 (Vol. 91, No. 1, pp. 53-63). North-Holland.
\bibitem{31} Mahmudov N, Gupta V. On certain $q$-analogue of Sz\'asz Kantorovich operators. Journal of Applied Mathematics and Computing. 2011 Oct 1;37(1):407-19.
\bibitem{32} Mirakjan GM. Approximation of continuous functions with the aid of polynomials. InDokl. Acad. Nauk SSSR 1941 (Vol. 31, pp. 201-205).
\bibitem{33} Mishra VN, Khatri K, Mishra LN. Statistical approximation by Kantorovich-type discrete $q$-Beta operators. Advances in Difference Equations. 2013 Dec 1;2013(1):345.
\bibitem{34} Mishra VN, Yadav R. Some estimations of summation-integral-type operators. Tbilisi Mathematical Journal. 2018;11(3):175-91.
\bibitem{35} Neer T, Agrawal PN. Quantitative-Voronovskaya and Gr\"uss-Voronovskaya type theorems for Sz\'asz-Durrmeyer type operators
blended with multiple Appell polynomials. Journal of inequalities and applications. 2017 Dec;2017(1):244.
\bibitem{36} \"Ozarslan MA, Aktu\u{g}lu H. Local approximation properties for certain King type operators. Filomat. 2013 Jan 1;27(1):173-81.
\bibitem{37} \"Oks\"uzer \"O, Karslı H, Ye\c{s}ildal FT. Convergence rate of B\'ezier variant of an operator involving Laguerre polynomials of degree n. In AIP Conference Proceedings 2013 Oct 17 (Vol. 1558, No. 1, pp. 1160-1163). AIP.
\bibitem{38} \"Ozarslan MA, Duman O, Kaano\u{g}lu C. Rates of convergence of certain King-type operators for functions with derivative of bounded variation. Mathematical and Computer Modelling. 2010 Jul 1;52(1-2):334-45.
\bibitem{39}Shaw SY, Liaw WC, Lin YL. Rates for approximation of functions in $BV[a,b]$ and $DBV[a,b]$ by positive linear operators. Chinese Journal of Mathematics. 1993 Jun 1:171-93.
\bibitem{40} Sz\'asz O. Generalization of S. Bernsteins polynomials to the innite interval. J. Res. Nat. Bur. Standards. 1950 Sep;45:239-
45.
\bibitem{41} Totik V. Approximation by Sz\'asz-Mirakjan-Kantorovich operators in $L_p(p > 1)$. Analysis Mathematica. 1983 Jun 1;9(2):147-67.
\bibitem{42} Ta\c{s}delen F, Akta\c{s} R, Altın A. A Kantorovich type of Sz\'asz operators including Brenke-type polynomials. In Abstract and Applied Analysis 2012 (Vol. 2012). Hindawi.
\bibitem{43} Varma S, Sucu S, İ\c{c}\"oZ G. Generalization of Sz\'asz operators involving Brenke type polynomials. Computers and Mathematics with Applications. 2012 Jul 1;64(2):121-7.
\bibitem{44} Walczak Z. On approximation by modified Sz\'asz-Mirakyan operators. Glasnik matemati\v{c}ki. 2002 Dec 1;37(2):303-19.










\end{thebibliography}
\end{document}